\documentclass[12pt, twoside, reqno] {amsart}

\usepackage{texdraw}
\usepackage{amsfonts}
\usepackage{amssymb}
\usepackage{latexsym}
\usepackage{epsf,graphicx}
\usepackage{epsf,graphicx,latexsym,%
}

\usepackage{ulem}

\usepackage{color}
\definecolor{R}{rgb}{1,0,0}
\definecolor{G}{rgb}{0,1,0}
\definecolor{B}{rgb}{0,0,1}
\definecolor{bl}{rgb}{0,0,0}
\definecolor{Wh}{rgb}{1,1,1}

\newcommand{\be} {\begin{eqnarray}}
\newcommand{\ee} {\end{eqnarray}}
\newcommand{\bep} {\begin{eqnarray*}}
\newcommand{\eep} {\end{eqnarray*}}

\newcommand {\Hol}{\mathop{\rm Hol}\nolimits}
\renewcommand {\Im}{\mathop{\rm Im}\nolimits}
\renewcommand {\Re}{\mathop{\rm Re}\nolimits}

\newcommand {\GP}{\mathcal{G}_{\alpha}(\Pi)}

\newcommand{\R}{{\mathbb R}}
\newcommand{\N}{{\mathbb N}}

\newcommand{\C}{{\mathbb C}}

\newcommand {\D}{\mathbb{D}}

\newtheorem{remar}{Remark}[section]
\newtheorem{examp}{Example}[section]
\newtheorem{defin}{Definition}[section]
\newtheorem{corol}{Corollary}[section]
\newtheorem{propo}{Proposition}[section]
\newtheorem{theorem}{Theorem}[section]
\newtheorem{lemma}{Lemma}[section]
\newtheorem{remark}{Remark}[section]

\newcommand{\rema}{\begin{remar}\rm}
\newcommand{\erema}{$\blacktriangleright$\end{remar}}

\newcommand{\exa}{\begin{examp}\rm}
\newcommand{\eexa}{$\blacktriangleright$\end{examp}}

\def\lwvec(#1 #2){\linewd 0.1
           \lvec(#1 #2)
           \linewd 0.05}

\begin{document}

\title{Analyticity of semigroups on the right half-plane}

\author[M. Elin]{Mark Elin}

\address{Department of Mathematics,
         Ort Braude College,
         Karmiel 21982,
         Israel}

\email{mark$\_$elin@braude.ac.il}

\author[F. Jacobzon]{Fiana Jacobzon}

\address{Department of Mathematics,
         Ort Braude College,
         Karmiel 21982,
         Israel}

\email{fiana@braude.ac.il}

\begin{abstract}

This paper is devoted to the study of semigroups of composition
operators and semigroups of holomorphic mappings. We establish
conditions under which these semigroups can be extended in their
parameter to sector given a priori. We show that the size of this
sector can be controlled by the image properties of the
infinitesimal generator, or, equivalently, by the geometry of the
so-called associated planar domain. We also give a complete
characterization of all composition operators acting on the Hardy
space $H^p$ on the right half-plane.

{\footnotesize Key words and phrases: holomorphic mapping,
semigroup, composition operator, Hardy space, function convex in
one direction.

}

\end{abstract}
\maketitle

\section{Introduction}

In this paper, we study semigroups of holomorphic self-mappings of
the right half-plane $\Pi=\{z\in\C:\Re z >0\}$ and semigroups of
bounded linear operators acting on the Hardy spaces $H^p(\Pi), \
p>1$, and their interaction.

Recall that a one-parameter family $\mathbf{T}:=\{{T}(t),t \geq
0\}$ of bounded linear operators on a Banach space $X$ is said to
be a {\it $C_0$-seimgroup} if it satisfies
\begin{itemize}
\item[(a)] ${T}(t){T}(s)= T(t+s)$ for all $t,s\geq 0$,

\item[(b)] $\lim\limits_{t\to0^+}T(t)=T(0)=I$, where $I$ is the
identity operator on $X$.
\end{itemize}

For each $C_0$-semigroup $\mathbf{T}$ on $X$, there exist scalars
$a \geq 0$ and $M \geq 1$ such that $\|T(t)\|\leq Me^{at}$ for all
$t \geq 0$. If $\|T(t)\|\leq 1$ for all $t \geq 0$, then the
$\mathbf{T}$ is said to be a contractive semigroup.

Denote $\mathcal{D}:=\left\{x \in X: \lim\limits_{t \to
0}\dfrac{T(t)x-x}{t} \mbox{ exists} \right\}$. The linear operator
$\Gamma$ defined on $\mathcal{D}$ by
\[\Gamma
x:=\lim\limits_{t \to 0}\frac{T(t)x-x}{t} =\frac{dT(t)x}{dt}\] is
called the {\it infinitesimal generator} of the semigroup
$\mathbf{T}$. The domain $\mathcal{D}$ is a dense linear subspace
of $X$ (see, for example \cite{Pa}).

Let $X^*$ be the dual space to $X$. Recall that a linear operator
$\Gamma$ (in general, unbounded) defined on a dense subspace
$\mathcal{D}\subset X$ is called {\it dissipative} if for all $\,
x\in\mathcal{D}\,$, $\|x\|=1$, there exists $\, x^*\in X^* \,$
with $\,\langle x,x^*\rangle=$, $\|x^*\|=1$ which satisfies
\[
\Re\langle\Gamma x,x^*\rangle \leq 0.
\]
The functional $x^*$, is not unique in general (see, for example,
\cite{Pa}); it is often called a support functional at the point
$x$.

The classical Lummer--Phillips theorem (see \cite{Pa},
\cite{R-S1}) asserts that {\it an operator $\Gamma$ generates a
contractive semigroup if and only if $\Gamma$ is dissipative and
$\lambda I-\Gamma $ is surjective for some (hence for all)
$\lambda > 0$.} Moreover, if $\lambda I-\Gamma $ is surjective for
every $\lambda
> 0$, then $\mu I-\Gamma $ is surjective for every $\mu \in \C$
such that $\Re \mu > 0$ (see, for example, \cite[Remark 5.4]{Pa}).

An important class of bounded linear operators which links to
Complex Analysis and attracts special attention consists of
composition operators. Each holomorphic self-mapping $F\in\Hol(D)$
of a domain $D\subset\C$ induces a composition operator $C_F$
defined by
\[
C_F:\varphi\mapsto \varphi\circ F
\]
on the Frech\`et space $\Hol(D,\C)$. Composition operators make up
a wide class of explicit examples in Operator Theory. Quoting
\cite{CCC-BMC}, ``\ldots they are surprisingly general and occur
in settings other than the obvious ones.'' In order to consider
composition operators as a part of Operator Theory, it is
necessary to study properties of their restrictions to different
Banach spaces $X\subset\Hol(D,\C).$

We proceed by reviewing several concepts and facts related to
semigroups of holomorphic self-mappings of a domain $D\subset\C$.
Throughout this paper, we denote the set of holomorphic functions
on $D$ taking values in a set $U$ by $\Hol(D,U)$ and let
$\Hol(D):=\Hol(D,D)$.

\textit{A one-parameter continuous semigroup} (semigroup, for
short) on $D$ is  a family $\mathbf{F}=\left\{F_t\right\}_ {t\geq
0} \subset\Hol(D)$ such that
\begin{itemize} \item[(i)] $F_{t}(F_{s}(z))= F_{t+s}(z)$ for all
$t,s\geq 0$ and $z\in D$,

\item[(ii)] $\lim\limits_{t\rightarrow 0^+}F_{t}(z)=z$ for all
$z\in D$.
\end{itemize}
It follows by \cite{B-P} that each semigroup on a simply connected
domain $D$ is differentiable with respect to $t\in\R^{+}=[0,\infty
)$. Thus, for each one-parameter continuous semigroup, the limit
\[
\lim_{t\rightarrow 0^{+}}\frac{F_{t}(z)-z}{t}:=f(z),\quad z\in D,
\]
exists and defines a holomorphic function $f\in\Hol(D,\C)$ called
the \textit{(infinitesimal) generator of} $\mathbf{F}$. Moreover,
the function $u:\mathbb{R}^{+}\times D\to \C$ defined by
$u(t,z):=F_{t}(z)$, is the unique solution of the Cauchy problem
\begin{equation}\label{cauchy}
\left\{
\begin{array}{l}
{\displaystyle\frac{\partial u(t,z)}{\partial t}} = f(u(t,z))
,\vspace{2mm}\\ u(0,z)=z,\quad z\in D.
\end{array}
\right.
\end{equation}

Since all simply connected domains are holomorphically equivalent,
one usually considers semigroups on either the right half-plane
$\Pi$, or the open unit disk $\D$ (see \cite{SD}). It is well
known that for a semigroup $\mathbf{F}$ on $\Pi$ that contains
neither the identity nor an elliptic automorphism, there exists a
unique $\tau\in\overline {\Pi}\cup\{\infty\}$, called the
Denjoy--Wolff point of $\mathbf{F}$, to which $F_t(z)$ converges
as $t \to \infty$ for all $z \in \Pi$ (see \cite{SD} for details).

In this paper, we study semigroups on the right half-plane$\Pi$
with the Denjoy--Wolff point at $\infty$. The generators of such
semigroups satisfy $\Re f(z)\ge 0$; and by Julia's Lemma, $\Re
F_t(z)$ is a non-decreasing function in $t$ for $t\ge0$ (see, for
example, \cite{SD}).

Clearly, each semigroup $\mathbf{F}= \{F_t\}_{t\geq
0}\subset\Hol(D)$ induces the semigroup $\mathbf{C}=\{C_t,t
\ge0\}$ of composition operators defined by
\begin{equation}
C_F:\varphi\mapsto \varphi\circ F_t
\end{equation}
on the Frech\`et space $\Hol(D,\C)$. Although composition
operators have a long history and diverse applications (see, for
example, \cite{CCC-BMC} and \cite{SJH-93} and references therein),
interest in semigroups of composition operators essentially began
with the paper \cite{B-P} by Berkson and Porta and the review
\cite{Sis} by Siskakis. During the last two decades, various
properties of semigroups of composition operators have been
studied by many mathematicians for different Banach spaces
$X\subset\Hol(D,\C)$ (such as the Dirichlet space, the Hardy
spaces, VMOA etc.) in the case $D=\mathbb{D}$, the open unit disk;
see, for example, \cite{SJH-93, Sis, A-C-P1, B-others}. In
particular, the problem of finding conditions for analytic
extension of semigroups of composition operators mentioned above
was considered in \cite{A-C-P, E-S-T}.

However, relatively little is known about semigroups of
composition operators on Banach spaces of holomorphic functions in
the right half-plane $\Pi$. It turns out that properties of even
single composition operators on Banach spaces of holomorphic
functions on $\Pi$ are different from analogous properties of
composition operators on $\D$. In particular, each composition
operator on the Hardy space $H^p(\D)$, is bounded, but the
composition operator $C_F$  on $H^p(\Pi)$ is bounded if and only
if the inducing function $F$ has a finite angular derivative at
infinity; see \cite{E-J-12}. Also, there is no compact composition
operator on $H^p(\Pi)$; see \cite{Mat1} and \cite{C-HD03}.

The paper is organized as follows. Section 2 contains auxiliary
results that are partially equivalent to some known statements
obtained in differing settings. We prove them for the sake of
completeness.

Sections 3 and 4 contain the main results. Namely, in Section~3,
we study semigroups of holomorphic self-mappings of $\Pi$ whose
generators have argument controlled by functions given a priori.
We find sectors to which such a semigroup $\mathbf{F}$
$=\left\{F_t\right\}_ {t\geq 0}$ as well as its restrictions to
invariant subsets can be analytically extended in the semigroup
parameter. We also describe domains (depending on $z\in\Pi$) which
contain the semigroup trajectory $\left\{F_t(z),\ t\geq
0\right\}$.

In Section 4, we deal with semigroups of composition operators on
the Hardy space $H^p(\Pi),\ \ p\in(1,\infty)$. Recall that
\[
H^p(\Pi)=\left\{\varphi\in\Hol(\Pi,\C): \|\varphi \|=\sup_{x>0}
\left(\frac1\pi \int_{-\infty}^{\infty} |\varphi(x+iy)|^p dy
\right)^{1/p} <\infty\right\}.
\]
Thus, we amplify results obtained earlier in \cite{Arv} and
\cite{A-C-P}. For a given semigroup of composition operators on
$H^p(\Pi)$, the main result describes its analytic extension to a
sector. Roughly speaking, principal parts of
Theorems~\ref{sector-basic} and \ref{th-operator-sem} can be
combined as follows. For an open sector $\Omega\subset\C$ with
vertex at zero, let $\widetilde{\Omega}=\left\{w\in\C: {\Re(zw)>0}
\ \mbox{ for all }{z\in\Omega}\right\}$. We prove that under some
mild conditions on an infinitesimal generator $f\in\Hol(\Pi)$ and
a sector $\Omega$, {\it the following assertions are equivalent:

\begin{itemize}
\item[$\bullet$] the image $f(\Pi)$ is contained in $\widetilde
\Omega$;

\item[$\bullet$] the semigroup $\mathbf{F}$ generated by $f$ can
be analytically extended in its parameter to the sector $\Omega;$

\item[$\bullet$] the conformal mapping onto the planar domain
(associated with $\mathbf{F}$) is convex in every direction in
$\Omega$;

\item[$\bullet$] the semigroup $\mathbf{T}$ of composition
operators induced by $\mathbf{F}$ can be analytically extended in
its parameter to the sector $\Omega$;

\item[$\bullet$] for each $\zeta \in \Omega$, the linear operator
$\zeta f\frac{\partial}{\partial z}$ generates a contractive
semigroup.
\end{itemize} }
Additionally, Theorem~\ref{th-charact-comp-operator}, gives a
criterion for a bounded linear operator acting on $H^p(\Pi)$ to be
a composition operator.

We mention again an essential difference between semigroups of
composition operators on Hardy spaces on $\D$ and those on Hardy
spaces on $\Pi$. Semigroups preserving the origin in $\D$ as well
as induced composition semigroups on Hardy spaces were considered
in \cite{E-S-T}. One of the results of that work implies that any
semigroup of holomorphic self-mappings restricted to invariant
subsets (which are disks of radii $r<1$) can be analytically
extended to a sector, and the sector is wider the smaller $r$ is.
Consequently, the induced composition semigroup can be extended to
a wider sector too. However, each analytic semigroup of bounded
composition operators on $H^p(\Pi)$ is induced by a semigroup of
parabolic type; see Corollary \ref{last}. In addition, as we show,
the restriction of a semigroup to invariant subsets does not
increase the size of the sector of analyticity in general.

\section{Auxiliary results}
\setcounter{equation}{0}

\subsection{Functions convex in one direction}

Let $\mathbf{F}=\{F_t\}_{t\ge0}\subset\Hol(\Pi)$ be a semigroup.
It is known that in the case $\mathbf{F}$ has no interior fixed
points, Abel's functional equation
\begin{equation}\label{abel1}
h\left(F_t(z)\right)=h(z)+t
\end{equation}
has a unique up to a constant solution $h\in\Hol(\Pi,\C)$ which is
a univalent function (see, for example, \cite{E-S-book, SD}). The
solution normalized by $h(1)=0$ is called the associated
K{\oe}nigs function of $\mathbf{F}.$ Differentiating
\eqref{abel1}, we conclude that
\begin{equation}\label{h-1}
h'(z)f(z)=1.
\end{equation}

In connection with equation \eqref{abel1}, recall that a domain
$D\subset\C$ is convex in the positive direction of the real axis,
if $w+t\in D$ for all $w\in D$ and $t\ge0$. The study of univalent
mappings onto such domains was initiated by Robertson and
continued by many mathematicians; see \cite{goodman} and
\cite{E-S-book} for appropriate references and historical details.

\begin{defin}[cf., {\upshape{\cite[Definition 3.5]{E-S-book}}}]\label{def-convex-one-dir}
We say that a univalent function $h\in\Hol(\Pi,\C)$ normalized by
$h(1)=0$ is convex in direction $e^{i\theta}$, $\ \theta\in\R$, if
for each $z \in \Pi$,
\begin{equation}
h(z)+te^{i\theta} \in h(\Pi) \text{ for each }t>0 \quad
\mbox{and}\quad \lim \limits_{t \to
\infty}h^{-1}\left(h(z)+te^{i\theta}\right)=\infty.
\end{equation}
We denote the class of all such functions by $\Sigma(\Pi,\theta)$
and let $\Sigma(\Pi):=\Sigma(\Pi, 0)$.
\end{defin}

\begin{remark}\label{rem-outside_point}
It follows immediately from Definition~\ref{def-convex-one-dir}
that ${h\in\Sigma(\Pi,\theta)}$ if and only if
$\xi-te^{i\theta}\not\in h(\Pi)$ for all $\xi\not\in h(\Pi)$ and
$t\ge0$.
\end{remark}

In the sequel, we use the following characterization of functions
univalent in  $\Pi$ and convex in the positive direction of the
real axis.

\begin{lemma}\label{lem-conv-cond}
A univalent function $h\in \Hol (\Pi, \C)$ is convex in the
positive direction of the real axis if and only if $\Re
h^\prime(z)\geq 0.$
\end{lemma}

\begin{proof}
Let $C\in\Hol(\D,\Pi)$ be the Cayley transform given by
$C(w)=(1+w)/(1-w)$. The function $g=h\circ C$ is analytic in $\D$.
Differentiating it, we obtain
\[
g'(w)=h'(C(w))C'(w)=\frac{2h^\prime(C(w))}{(1-w)^2}.
\]
By \cite[Lemma 3.8]{E-S-book}, $g$ is convex in the positive
direction of the real axis if and only if $\Re g^\prime(w)(1-w)^2
\geq 0.$ Thus $h$ is convex in the positive direction of the real
axis if and only if $\Re h^\prime(z)\geq 0, \, z\in \Pi.$
\end{proof}

In what follows, we deal with functions which are convex in every
direction in some sector. Namely, for $0<\theta_1,\theta_2$ with
$\theta_1+\theta_2\leq \pi,$ we define
\begin{equation}\label{conv-in-cone}
\Sigma(\Pi,\theta_1,\theta_2)= \bigcap_{
\theta\in(-\theta_1,\theta_2)}\Sigma(\Pi,\theta)
\end{equation}
and
\begin{equation}\label{sector}
\Omega(\theta_1,\theta_2) = \left\{\zeta: -\theta_1< \arg \zeta <
\theta_2 \right\}.
\end{equation}
Clearly, each function $h \in \Sigma(\Pi,\theta_1,\theta_2)$
satisfies
\begin{equation}\label{in-cone}
h(z)+ \overline{\Omega(\theta_1,\theta_2)}\subset h(\Pi)
\end{equation}
for all $z\in\Pi$. We complete this fact as follows (cf.
\cite{Alecko}).

\begin{propo}\label{prop-conv-in-sector}
Let $h\in\Sigma(\Pi)$. Suppose that
\[
h(z)+ \Omega(\theta_1,\theta_2)\subset h(\Pi) \quad \text{for all
}z \in \Pi.
\]
Then $h \in \Sigma(\Pi,\theta_1,\theta_2).$
\end{propo}

\begin{proof}
Fix $\theta\in (-\theta_1,\theta_2)$. Note that the formula
\[
\Psi_t(z)=h^{-1}\left( h(z)+ te^{i\theta} \right)
\]
defines a one-parameter continuous semigroup
$\{\Psi_t\}_{t\ge0}\subset\Hol(\Pi).$ Note that this semigroup has
no interior fixed point. Indeed, if $z_0\in\Pi$ is its fixed point
$z_0\in\Pi$, then
\[
h(z_0)=h(z_0)+ te^{i\theta}
\]
for all $t>0$, which is impossible. Hence,
$\{\Psi_t\}_{t\ge0}\subset\Hol(\Pi)$ has the boundary
Denjoy--Wolff point $\zeta.$ We have to show that $\zeta=\infty$.

Let $a_0\not\in h(\Pi)$. By Remark~\ref{rem-outside_point},
$a_0-te^{i\beta}\not\in h(\Pi)$ for all $\beta \in
\left(-\theta_1, \theta_2\right)$ and then $a=a_0-1$ has a
neighborhood contained in $\C\setminus h(\Pi)$. Thus the function
$g\in\Hol(\D, \C)$ given by
\begin{equation}\label{g}
g(w)=\frac{1}{h\left(\frac{1+w}{1-w}\right)-a}
\end{equation}
is univalent and bounded in $\D$. Moreover, $g$ satisfies
\begin{itemize}
\item[(i)] $0\in\partial g(\D)$;

\item[(ii)] $\dfrac{g(w)}{1-te^{i\theta}g(w)} \in g(\D)$ whenever
$t\geq0$ and $\theta\in [-\theta_1,\theta_2]$;

\item[(iii)] $\lim\limits_{t \to \infty} g^{-1}\left(
\dfrac{g(w)}{1-tg(w)}\right) =1$.
\end{itemize}

By the Carath\'eodory Theorem (see, for example, \cite{C-L}),  $g$
induces the one-to-one correspondence between points of the unit
circle $\partial\D$ and prime ends of $g(\D)$. By property (iii),
the impression $I(P(1))$ of the prime end $P(1)$ contains the
point $\zeta_0=0$, which is its unique accessible point.

Let $\{C_n\}$ be a chain of crosscuts defining $P(1)$. For a fixed
$w\in\D$, we can assume without loss of generality that each $C_n$
separates $g(w)$ and $\zeta_0$. By property (ii),
\[
G_w:=\left\{\dfrac{g(w)}{1-te^{i\theta}g(w)}:\ t\geq0 \; \theta\in
[-\theta_1,\theta_2]\right\}\subset g(\D).
\]
Thus, each crosscut $C_n$ intersects every path joining $g(w)$ and
$\zeta_0$ that, with the exception of its end point $\zeta_0$,
lies inside $G_w$. In particular, each $C_n$ intersects the
trajectory $\Gamma_w =\left\{ \dfrac{g(w)}{1-te^{i\theta}g(w)} :t
\geq 0 \right\}$, which tends to $0$ as $t\to\infty$. So this
trajectory tends to the same unique accessible point $\zeta_0=0$
of the prime end $P(1)$. Bearing in mind \eqref{g}, we see that
the last fact coincides with the assertion. \end{proof}

\subsection{Semigroup generators in the right half-plane}

Here we present some necessary and sufficient conditions for a
function $f\in\Hol(\Pi,\C)$ to be a semigroup generator.

\begin{propo}\label{propo-flow-inv}
Let $f\in\Hol(\Pi,\C)$. Then $f$ is a semigroup generator if and
only if
\be\label{flow-inv} %
\Re\left(f(z)\frac{\overline{z}-1}{z+1} \right) \le \Re z\cdot
\Re\left(f(1)\frac{\overline{z}-1}{\overline{z}+1} \right)
\ee%
for all $z\in\Pi$.
\end{propo}

\begin{proof}
To prove the necessity of \eqref{flow-inv}, suppose that $f$
generates a semigroup $\{F_t \}_{t\ge0}\subset\Hol(\Pi)$. Let
$C\in\Hol(\D,\Pi)$ be the Cayley transform of $\D$ onto $\Pi$
defined by $w\mapsto\dfrac{1+w}{1-w}\,$. It is easy to see that
the family $\{\Phi_t\}_{t\ge0}$ defined by $\Phi_t=C^{-1}\circ
F_t\circ C$ forms a semigroup  on $\D$. We find its generator $g$
by differentiation:
\[
g(w)=\lim_{t\rightarrow 0^{+}}\frac{\Phi_{t}(w)-w}{t}=
\left(C^{-1}\right)'(C(w))f(C(w))=\frac{(1-w)^2f(C(w))}2\,.
\]
Since $g$ is a semigroup generator in $\D$, it satisfies the
flow-invariance condition $\Re \left(g(w)\overline{w}\right)\le
(1-|w|^2) \Re\left(g(0)\overline{w}\right)$ (see \cite[Proposition
3.5.4]{SD} and \cite[Theorem 2.3]{E-S-book}\footnote{In
\cite[Theorem 2.3]{E-S-book}, the opposite inequality sign
appears.  This is due to a slightly different definition.}). Thus,
\[
\Re \left((1-w)^2 f(C(w))\overline{w}\right)\le (1-|w|^2)
\Re\left(f(1)\overline{w}\right).
\]
Substituting $w=C^{-1}(z)=\dfrac{z-1}{z+1}$, we obtain the
required inequality.

Conversely, suppose that $f$ satisfies \eqref{flow-inv}.

Then the function $g\in\Hol(\D,\C)$ defined by
$g(w)=\frac{1}{2}(1-w)^2f(C(w))$ satisfies the flow-invariance
condition, and hence generates a semigroup $\{\Phi_t\}_{t\ge0}$ on
$\D$. Then the functions $F_t=C \circ \Phi_t\circ C^{-1}$ form a
semigroup on $\Pi$. Differentiating $\{F_t\}_{t\geq 0}$ with
respect to $t$ at $0$, one sees that $\{F_t\}_{t\geq 0}$ is
generated by~$f$.
\end{proof}

\begin{propo}\label{prop-DW}
Let $f\in \Hol(\Pi,\C)$ be a semigroup generator in $\Pi$. If
$f'(\infty):=\angle\lim\limits_{z\to\infty}\dfrac{f(z)}{z+1}$
exists and $f'(\infty) \ge 0$, then $\Re f(z)\ge0$ for all
$z\in\Pi$.
\end{propo}

\begin{proof}
Denote by $\mathbf{F}=\{F_t \}_{t\ge0} \subset \Hol(\Pi)$ the
semigroup generated by $f$ and by $\mathbf{\Phi}=\{\Phi_t\}_
{t\ge0}\subset\Hol(\D)$ the semigroup defined by $\Phi_t= C^{-1}
\circ F_t\circ C$. The generator $g\in\Hol(\D,\C)$ of
$\mathbf{\Phi}$ is given by $g(w)=\frac{1}{2}(1-w)^2f(C(w)).$ If
$\angle\lim\limits_ {z\to\infty}\frac{f(z)}{z+1}=a\ge0$, then
\[
\angle\lim_{w\to 1}\dfrac{g(w)}{w-1}= \angle\lim_{w\to
1}\dfrac{(w-1)f(C(w))}{2}=\angle\lim_{z\to\infty}
\frac{-f(z)}{z+1}=-a\le0.
\]
Hence, by \cite[Theorem 2.10]{E-S-book}, $w=1$ is the
Denjoy--Wolff point of $\mathbf{\Phi}$. In this case, $\mathbf{F}$
has the Denjoy--Wolff point at $\infty$. By Julia's lemma, $\Re
F_t(z)$ is a non-decreasing function in $t$. Now, by the
definition of the semigroup generator, we conclude that $\Re
f(z)\ge0$ for all $z\in\Pi$.
\end{proof}

\section{Semigroups of holomorphic mappings}\label{sect_hp}
\setcounter{equation}{0}

In this section, we find criteria for a semigroup of holomorphic
self-mappings of the right half-plane $\Pi$ to admit an analytic
extension in its parameter to a given sector. It follows from
these criteria that semigroups having repelling fixed points, as
well semigroups of hyperbolic type, cannot be extended to any
sector. Moreover, the same control functions that govern the size
of the sector of analyticity enable us to localize the semigroup
trajectories.

\begin{theorem}\label{sector-basic}
Let $f \in \Hol(\Pi,\C)$ be a generator of the one-parameter
continuous semigroup ${\mathbf{F}=\{F_{t}\}_{t\geq
0}\subset\Hol(\Pi)}$
 with the Denjoy--Wolff point at $\infty$,
and let $h$ be its associated K{\oe}nigs function. For all
$\theta_1,\theta_2>0$ such that $\theta_1+\theta_2\leq \pi$, the
following are equivalent:
\begin{itemize}
\item[{\rm(i)}] $\mathbf{F}$ can be analytically extended to a
semigroup $\{F_{\zeta}\}_{\zeta \in \Omega(\theta_1,\theta_2)}$
such that for each  $z \in \Pi$, the function $F_{\zeta}(z)$ tends
to $\infty$ along every ray in $\Omega(\theta_1,\theta_2);$

\item[{\rm(ii)}] the function $f$ satisfies
\begin{equation}\label{gen-in-sector}
-\frac{\pi}{2}+  \theta_1\leq\arg f(z) \leq \frac{\pi}{2}-
\theta_2,  \quad z \in \Pi;
\end{equation}

\item[{\rm(iii)}] $h \in \Sigma(\Pi,\theta_1,\theta_2).$
\end{itemize}
\end{theorem}

\begin{proof} {\it Step 1.} Suppose that assertion (i) holds.
Fix $z \in \Pi$. Since $\lim\limits_{\R\ni\zeta \to
0}\dfrac{F_\zeta(z)-z}{\zeta}=f(z)$ and the semigroup $\mathbf{F}$
is real-analytic in $t$ in a neighborhood of $t=0$, then by the
uniqueness theorem, $\dfrac{F_\zeta(z)-z}{\zeta} \to f(z)$ as
$\zeta \to 0$ along any path in $\C$. For $\theta\in
(-\theta_1,\theta_2)$, define semigroup $\{\Psi_t\}_{t\geq 0}$ by
$\Psi_t (z)= F_{te^{i\theta}}(z)$. Its generator $\psi$ is given
by
\[
\psi(z)=\left. \frac{\partial\Psi_t\left(z\right)}{\partial
t}\right|_{t= 0}=\left. \frac{\partial
F_{te^{i\theta}}\left(z\right)}{\partial t}\right|_{t=
0}=\lim\limits_{t \to
0^+}{e^{i\theta}}\dfrac{F_{te^{i\theta}}-z}{{te^{i\theta}}}={e^{i\theta}}f(z).
\]
Since $\lim\limits_{t\rightarrow\infty}F_{e^{i\theta}t}(z)=
\infty$, then $0\leq \Re\psi(z)=\Re{e^{i\theta}}f(z).$ Hence, by
the continuity of $e^{i\theta} f(z)$ in $\theta$, we have $
-\dfrac{\pi}{2}\leq\arg{e^{i\theta}}f(z)\leq\dfrac{\pi}{2}$ for
all  $\theta\in [-\theta_1,\theta_2]$, and (ii) follows.

{\it Step 2.} Assume now that (ii) holds. By formula \eqref{h-1},
$\arg h'(z)=-\arg f(z)$. Then, by \eqref{gen-in-sector}, we have
$-\dfrac{\pi}{2}\leq \arg h ^\prime(z) - \theta\leq \dfrac{\pi}{2}
\,$, from which it follows that ${\Re e^{-i\theta}h^\prime(z)\geq
0}$ for all $\theta\in [-\theta_1,\theta_2].$

By Lemma \ref{lem-conv-cond}, function $e^{-i\theta}h$ is convex
in the positive direction of the real axis, that is,
$e^{-i\theta}h(z)+t \in e^{-i\theta}h(\Pi)$ for all $t\geq 0$.
This implies that $h(z)+\zeta \in h(\Pi)$ for all $\zeta \in
\Omega(\theta_1,\theta_2).$ Assertion (iii) now follows, by
Proposition \ref{prop-conv-in-sector}.

{\it Step 3.} Assertion (iii) implies that
$F_\zeta(z):=h^{-1}(h(z)+\zeta)$ is well defined, analytic in
$\zeta \in \Omega(\theta_1,\theta_2)$, and assumes values in $\Pi$
for all $\zeta \in \Omega(\theta_1,\theta_2).$ For each pair
$\zeta_1,\zeta_2 \in \Omega(\theta_1,\theta_2)$, we have
$\zeta_1+\zeta_2 \in \Omega(\theta_1,\theta_2)$ and
\begin{eqnarray*}
&& F_{\zeta_1}\circ
F_{\zeta_2}(z)=F_{\zeta_1}\left(\sigma^{-1}(\sigma(z)+\zeta_2)
\right)= \\
&& \sigma^{-1}[\sigma\left(\sigma^{-1}(\sigma(z)+\zeta_2)
\right)+\zeta_1]=\sigma^{-1}(\sigma(z)+\zeta_2
+\zeta_1)=F_{\zeta_1+\zeta_2}(z).
\end{eqnarray*}
Finally, from Proposition \ref{prop-conv-in-sector}, we conclude
that $${\lim\limits_{t \to \infty}F_{te^{i\theta}}=\lim\limits_{t
\to \infty} h^{-1}\left( h(z)+ te^{i\theta} \right) =\infty}$$ for
all $z\in\Pi$ and $\theta\in [-\theta_1,\theta_2]$, which
establishes (i).
\end{proof}

\begin{examp}\label{exam1}
Consider the affine semigroup $\mathbf{F}=
\left\{F_t\right\}_{t\geq 0}\subset\Hol(D)$, where $F_t(z)=z+At$
for some $A\in \Pi$. This semigroup is generated by the constant
function $f(z)=A$. According to the equivalence of (ii) and (iii)
of Theorem~\ref{sector-basic}, $\mathbf{F}$ has an analytic
extension to the sector $\Omega\left(\frac{\pi}{2}+\arg
A,\frac{\pi}{2}-\arg A\right)$. Indeed, let $F_\zeta(z)=z+A\zeta$.
Then $\Re F_\zeta (z)>0$ for all $z\in\Pi$ whenever $\zeta \in
\Omega\left(\frac{\pi}{2}+\arg A,\frac{\pi}{2}-\arg A\right)$. The
same conclusion also follows from the equivalence of (iii) and
(i).
\end{examp}

The geometrical constructions used in Theorem~\ref{sector-basic}
imply that if a semigroup $\mathbf{F}$ has an analytical extension
to a sector, then $\mathbf{F}$ must be of parabolic type. Indeed,
in the hyperbolic case, the planar domain $h(\Pi)$ is completely
contained in a horizontal strip of finite width (see \cite{C-DM}),
so condition (iii) of Theorem \ref{sector-basic} is not satisfied.

By contrast, as in Example~\ref{exam1}, there are parabolic type
semigroups which admit an analytic extension to the half-plane
$\{\zeta: |\arg\zeta|<\frac\pi2\}$. On the other hand, there exist
parabolic type semigroups which cannot be extended. For example,
each semigroup of automorphic type satisfies
${d_{h(\Pi)}(w)=\lim\limits_{s\to \infty}\delta_ {h(\Pi)}(w +
s)<\infty}.$ Hence its planar domain $h(\Pi)$ is contained in a
horizontal half-plane. So no sector $\Omega\supset \R^+$ is
contained in $h(\Pi)$, i.e., condition (iii) of Theorem
\ref{sector-basic} is not satisfied.

Moreover, the planar domain of a semigroup that has a regular
repelling point contains a strip of the final maximal width; see
\cite{C-DM} and \cite{E-S-book}. Thus, condition (iii) of
Theorem~\ref{sector-basic} cannot be satisfied, and the semigroup
has no analytical extension to a sector.

\begin{corol}\label{corol-1}
Let ${\mathbf{F}=\{F_{t}\}_{t\geq 0}\subset\Hol(\Pi)}$ be a
one-parameter continuous semigroup with the Denjoy--Wolff point at
$\infty$. Suppose that $\mathbf{F}$ can be extended to a semigroup
$\{F_{\zeta}\}_{\zeta\in\Omega}$ which is analytic in some sector
$\Omega$ and such that $F_{\zeta}(z)\to\infty$ along every ray in
$\Omega$ for each $z\in\Pi$. Then $\mathbf{F}$ is of parabolic
type and has no boundary regular repelling points.
\end{corol}

\begin{proof}
Note that $\mathbf{F}$ has no interior fixed point since it has
the  Denjoy--Wolff point at $\infty$. Denote the generator of
$\mathbf{F}$ by $f \in \Hol(\Pi)$. In the hyperbolic case, $\Re
f(z)\ge 0$ and ${\angle\lim\limits_{z \to \infty}
\frac{f(z)}{z}=a>0}$. Hence, for every $\theta\ne0$, the function
$e^{i\theta}f(z)$ is not a semigroup generator, and condition (ii)
fails. Therefore, $\mathbf{F}$ must be of parabolic type.

Suppose by contrary, that $\mathbf{F}$ has a boundary repelling
fixed point $ib\in
\partial\Pi,\ b\in\R$. Then $f(z)=(z-ib)(f^\prime(ib)+\rho(z)),$
where $f^{\prime}(ib)>0$ and $\angle\lim\limits_{z\to ib}
\rho(z)=0$. Setting $z_r=ib+re^{i\theta}$ for a fixed
$\theta\in(-\pi/2,\pi/2)$, we see that for small enough $r$, $\arg
f(z_r)=\theta+\arg(f'(ib)+\rho(z_r))$ is close to $\theta$. By
Theorem~\ref{sector-basic} (ii), this contradicts the analytic
extendibility of $\mathbf{F}$.
\end{proof}

Note that the parabolic group $\{F_t\}_{t\in\R}$, $F_t(z)=z+iqt$
with $q\in\R$ admits an analytic extension to the horizontal
half-plane $\{\zeta: q\Im\zeta\le0 \}$.

Corollary \ref{corol-1} implies the following fact.

\begin{corol}\label{corol-2}
Let  $\Omega$ be a sector which is not contained in a horizontal
half-plane. There is no group
$\mathbf{F}=\{F_t\}_{t\in\R}\subset\Hol(\Pi)$ that can be
analytically extended in its parameter to $\Omega$.
\end{corol}

\begin{proof}
If such a group exists, it cannot be of elliptic type. Hence it
has the Denjoy--Wolff point $\tau\in\partial\Pi\cup \{\infty\}$.
If $\tau\neq\infty$, then the group $\{G_t\}_{t\in\mathbb{R}}$
defined by
\[
G_t(z)= \frac1{F_t\left(\frac1z +\tau\right)-\tau}
\]
has   the Denjoy--Wolff point $\tilde\tau=\infty$. By
Corollary~\ref{corol-1}, $\{G_t\}$ is of parabolic type, that is,
$G_t(z)=z+iqt$ for some $q\in\R$. Clearly, there is no sector
$\Omega$ such that $G_\zeta(\Pi)\subset\Pi$ for all
$\zeta\in\Omega$.
\end{proof}

We now turn to semigroup generators whose arguments are controlled
by functions given a priori.

\begin{theorem}\label{thm-ector-angle-general}
Let $\gamma_1, \gamma_2:\R^+ \to (0,\frac\pi2)$ be decreasing
functions. Let $\mathbf{F}=\{F_{t}\}_{t\geq 0}\subset\Hol(\Pi)$ be
a semigroup generated by $f \in \Hol (\Pi,\C)$ which satisfies
\begin{equation}\label{control-gen-condition}
-\gamma_1(\Re z)\leq \arg f(z)\leq \gamma_2(\Re z)\quad \mbox{for
all}\quad z\in\Pi.
\end{equation}
Then for each $k>0$, the semigroup
$\mathbf{F}^k=\{\Phi_{t}\}_{t\geq 0}$ defined by
$\Phi_{t}(z)=F_{t}(z+k)-k$ can be analytically extended in $t$ to
the sector
$\Omega\left(\frac\pi2-\gamma_1(k),\frac\pi2-\gamma_2(k)\right)$.
In particular, if $\lim\limits_{s \to \infty} \gamma_1(s)=
\lim\limits_{s \to \infty} \gamma_1(s)=0,$ then for each
$\varepsilon>0$, there exists $k>0$ such that $\mathbf{F}^k$ has
an analytic extension to the sector
$\Omega\left(\frac\pi2-\varepsilon, \frac\pi2-\varepsilon\right).$
\end{theorem}

The result is a direct consequence of the equivalence
(ii)$\Leftrightarrow$(i) of Theorem~\ref{sector-basic}.

In fact, the semigroups $\mathbf{F}^k$ that appear in
Theorem~\ref{thm-ector-angle-general} are the restrictions of the
original semigroup $\mathbf{F}$ to sub-half-planes of $\Pi$, so
Theorem~\ref{thm-ector-angle-general} gives a criterion for the
analytic extendibility of a semigroup considered not on the whole
domain but on its invariant subsets. It follows from \cite{E-S-T}
that the restriction of a semigroup on $\D$ with the Denjoy--Wolff
point $\tau=0$ to an invariant subset (which is a disk of radius
$r<1$) can be analytically extended to a sector, and the sector is
wider the smaller $r$ is. In our case, the situation is different.
As the following example shows, in general, the restriction of a
semigroup to invariant subsets need not increase the size of the
sector to which the semigroup can be extended.

\begin{examp}\label{exam2}
Fix $\alpha\in(0,1)$, and consider the semigroup generator defined
by ${f(z)=z^{1-\alpha}}$. Obviously, $|\arg
f(z)|\le\frac{\pi(1-\alpha)}2$, for all $z\in\Pi$; hence, by
Theorems~\ref{sector-basic}--\ref{thm-ector-angle-general}, the
generated semigroup ${\mathbf{F}=\{F_t\}_{t\ge0}}$ can be extended
to the sector $\Omega\left(\frac{\pi\alpha}2, \frac{\pi\alpha}2
\right)$. This semigroup can be found directly as the unique
solution of the Cauchy problem
\[
\left\{
\begin{array}{l}
{\displaystyle\frac{\partial F_t(z)}{\partial t}} =
(F_t(z))^{1-\alpha},\vspace{2mm}\\
F_0(z)=z.
\end{array}
\right.
\]
So, $F_t(z)=\left(\alpha t+z^\alpha \right)^{1/\alpha}$. For each
$k>0$, the restriction of $\mathbf{F}$ to the half-plane $\{z:\Re
z>k\}$, or, more precisely, the semigroup
$\mathbf{F}^k=\{\Phi_{t}\}_{t\geq 0}\in\Hol(\Pi)$ defined by
$\Phi_{t}(z)=F_{t}(z+k)-k$, is generated by
$f^k(z)=(z+k)^{1-\alpha}.$ Therefore, $\mathbf{F}^k$ can be
analytically extended to the same sector
$\Omega\left(\frac{\pi\alpha}2, \frac{\pi\alpha}2 \right)$ but to
no larger sector.
\end{examp}

It turns out that the same control functions which prescribe in
Theorem~\ref{thm-ector-angle-general} the size of the sector of
analyticity for a semigroup provide a localization of semigroup
trajectories in the sense of differential inclusions.

\begin{propo}\label{th-image-domain}
Let $\{F_{t}\}_{t\geq 0}\in\Hol(\Pi)$ be a semigroup generated by
${f\in\Hol(\Pi,\C)}$. Suppose that there exist continuous
functions ${\gamma_1, \gamma_2:\R^+ \rightarrow
(0,\frac{\pi}{2})}$, such that
\begin{equation}\label{condition_gen}
-\gamma_1(\Re z)\leq \arg f(z)\leq \gamma_2(\Re z) \mbox{for all}
z \in \Pi.
\end{equation}
Then, for each initial point $z\in\Pi$, the trajectory $\{F_t(z),\
t\geq 0\}$ lies in the domain
\begin{equation}\label{domain of traj}
\left\{u+iv:\ u\geq \Re z , \  B_1(u)
 \leq v \leq  B_2(u)\right\},
\end{equation}
where
\begin{equation}\label{B-1-2}
\begin{array}{l}
B_1(u)=\Im z-\int\limits_{\Re z}^u\tan\gamma_1(s) ds, \\
B_2(u)=\Im z+\int\limits_{\Re z}^u \tan\gamma_2(s) ds.
\end{array}
\end{equation}
\end{propo}

\begin{proof} Fix $z\in\Pi$. Separating the real and imaginary parts, we write
\begin{equation*}
F_t(z)={u}_t(z) + i{v}_t(z)\quad \mbox{and}\quad f(z)={\phi}(z) +
i{\psi}(z).
\end{equation*}
Then $u_0=\Re z$, $v_0=\Im z$, $\frac{\partial u_t}{\partial
t}=\phi(u_t+iv_t)>0$, and $\frac{\partial v_t}{\partial
t}=\psi(u_t+iv_t).$ Therefore, along the semigroup trajectory
$\{F_t(z):t\geq 0\}$, the function $u_t$ is strictly increasing.
Hence, $v_t$ can be considered as a function of $u_t$.
Furthermore, $ \frac{dv_t}{du_t}=\tan\left(\arg
f(u_t+iv_t)\right).$ Consequently, by \eqref{condition_gen},
\begin{equation*}
-\tan\gamma_1(u_t)\leq \frac{dv_t}{du_t}\leq \tan\gamma_2(u_t).
\end{equation*}
Integrating this inequality with respect to $u_t$, we obtain
\begin{equation*}
-\int_{u_0}^{u_t}  \tan\gamma_1(s) ds \leq v_t(u_t)-v_0 \leq
\int_{u_0}^{u_t} \tan\gamma_2(s) ds
\end{equation*}
for all $t\ge 0$. The proof is complete.
\end{proof}

\begin{examp}\label{exam3}
Let $f\in\Hol(\Pi)$ be defined by $f(z)=\dfrac{z+a}{z+b}$, $0\le
a<b$. Since ${\lim\limits_{z\to\infty}\dfrac{f(z)}{z+ 1 }=0}$, $f$
generates a semigroup of parabolic type which is the unique
solution of the Cauchy problem
\[
\left\{
\begin{array}{l}
{\displaystyle\frac{\partial F_t(z)}{\partial t}} =
\dfrac{F_t(z)+a}{F_t(z)+b} \vspace{2mm}\\
F_0(z)=z, \, z\in \Pi,
\end{array}
\right.
\]
and hence satisfies the functional equation
\[
F_t(z) + (b-a)\log(F_t(z)+a)=t+z +(b-a)\log(z+a).
\]
This equation, being transcendental, does not allow localization
of the semigroup trajectories. Nevertheless, it is easy to see
that
\[
\left| \arg f(z) \right| \le \arctan\frac{b-a}{2\sqrt{(\Re z+a)
(\Re z+b)}}=:\gamma(\Re z)\,.
\]
Therefore, by Proposition~\ref{th-image-domain},
\[
\left|\Im F_t(z)-\Im z\right| \le \frac{b-a}2
\log\frac{\frac{a+b}2+\Re F_t(z)+\sqrt{(\Re F_t(z)+a)(\Re F_t(z)+b)}}{\frac{a+b}2+\Re z +
\sqrt{(\Re z+a)(\Re z+b)}}
\]
for all $z\in \Pi$ and $t\ge 0$.

Furthermore, according to Theorems~\ref{sector-basic} and
\ref{thm-ector-angle-general}, the semigroup $\mathbf{F}$
restricted to the half-plane $\{z:\Re z>k\}$ can be analytically
extended to the sector
$\Omega\left(\frac\pi2-\gamma(k),\frac\pi2-\gamma(k)\right)$. A
direct calculation shows that for all $\varepsilon>0$ and
$k>\dfrac{b-a-(a+b)\sin\varepsilon} {2\sin\varepsilon}$, the
restricted semigroup $\mathbf{F}^k$ has an analytic extension to
the sector $\Omega\left(\frac\pi2-\varepsilon,
\frac\pi2-\varepsilon\right).$
\end{examp}

Theorem~\ref{thm-ector-angle-general} can also be used to study
the important classes $\GP,\  \alpha>0,$ of semigroup generators
of the form
\begin{equation}\label{phi1a}
f(z)= A(z+1)^{1-\alpha} +\varrho(z), \quad \mbox{ where } \quad
\lim_{z\to\infty}\frac{\varrho(z)}{(z+1)^{1-\alpha}}=0
\end{equation}
and $A >0$. Note that for $f \in \GP$,
\begin{equation}\label{arglim}
\lim\limits_{\Re z\to\infty}\arg \frac{f(z)}{A(z+1)^{1-\alpha}}=0.
\end{equation}
These classes were first introduced in \cite{E-S-Y} for semigroup
generators in the open unit disk and then studied in more detail
in \cite{E-K-R-S, F-M2011}. In particular, it was shown that if $f
\in\GP$, then $\alpha \le 2$ and $|\arg A|\le \frac{\pi}{2}\,\min
\{\alpha, 2-\alpha\}$. For $0<\alpha<2$, we now consider the
slightly narrower class of generators for which $|\arg A|
<\frac{\pi}{2}\,\min \{\alpha, 2-\alpha\}$. The functions
$\delta_1,\delta_2$ defined on $(0,\infty)$ by
\begin{eqnarray}\label{delta}
\begin{array}{l}
\displaystyle \delta_1(s)=-\inf_{\Re z\ge s}\arg \frac{f(z)}
{A(z+1)^{1-\alpha}}, \\
\displaystyle \delta_2(s)=\sup_{\Re z\ge s}\arg
\frac{f(z)}{A(z+1)^{1-\alpha}}
\end{array}
\end{eqnarray}
are non-increasing and, by~\eqref{arglim}, tend to $0$ as
$s\to\infty$. Moreover, for every $z\in\Pi$,
\begin{eqnarray*}
(1-\alpha)\arg(1+z) +\arg A -\delta_1(\Re z) \le \arg f(z)\\
\le (1-\alpha)\arg(1+z) +\arg A +\delta_2(\Re z),
\end{eqnarray*}
that is, inequality~\eqref{control-gen-condition} holds with
\begin{eqnarray*}
\displaystyle \gamma_1(s)=\frac{|1-\alpha|\pi}2 -\arg A
+\delta_1(s), \\
\displaystyle \gamma_2(s)=\frac{|1-\alpha|\pi}2 +\arg A
+\delta_2(s).
\end{eqnarray*}
This establishes the following consequence of
Theorem~\ref{thm-ector-angle-general}.

\begin{corol}\label{th-sector-angle1}
Let $\{F_{t}\}_{t\geq 0}\in\Hol(\Pi)$ be a semigroup generated by
$f \in \GP,\ \alpha\in(0,2),$ such that \eqref{phi1a} holds with
$|\arg A| <\frac{\pi}{2}\,\min \{\alpha, 2-\alpha\}$. Then for
each $k>0$, the semigroup $\mathbf{F}^k=\{\Phi_{t}\}_{t\geq
0}\in\Hol(\Pi)$ defined by $\Phi_{t}(z)=F_{t}(z+k)-k$ can be
analytically extended to the sector
$\Omega\left(\theta_1(k),\theta_2(k)\right)$, where
\begin{eqnarray*}
\displaystyle \theta_1(k)=\frac{\pi}2\,\min \{\alpha, 2-\alpha\}
+\arg A
-\delta_1(k), \\
\displaystyle \theta_2(k)=\frac{\pi}2\,\min \{\alpha, 2-\alpha\}
-\arg A -\delta_2(k),
\end{eqnarray*}
and $\delta_1,\delta_2$ are defined by \eqref{delta}.
Consequently, for each sector $\widetilde{\Omega}$ such that
$$\overline{\widetilde{\Omega}}\setminus\{0\}\subset \Omega\left(\frac{\pi}2\,\min \{\alpha, 2-\alpha\} +\arg A,
\frac{\pi}2\,\min \{\alpha, 2-\alpha\} -\arg A \right),$$ there
exists $k>0$ such that $\mathbf{F}^k$ has an analytic extension
to~$\tilde\Omega$.
\end{corol}

\begin{remark}
Note that the union $\bigcup_ {\alpha \in (0,2]}\GP$ provides a
wide class of semigroup generators for which the so-called
Slope-problem has affirmative answer. This problem that was posed
in \cite{C-DM} and stayed open for about ten years, can be
formulated as follows. For $z\in \Pi$, is the set
$Slope^+(F_t(z))$ of all limit points of the curve $\left\{(t,\arg
F_t(z)): t\in (0,\infty)\right\}$ as $t\to\infty$ always a
singleton?  This question was answered negatively in general quite
recently in \cite{Bet} and \cite{C-DM-sl}.

Let $\{F_t\}_{t>0}$ and $f$ satisfy the hypotheses of
Proposition~\ref{th-image-domain}, and let $z\in\Pi$. Assume that
$\limsup\limits_{u \to \infty}\frac{B_1(u)}{u}>\liminf\limits_{u
\to \infty}\frac{B_2(u)}{u}$ , where $B_1$ and $B_2$ are defined
by \eqref{B-1-2}. Proposition \ref{th-image-domain} implies that
$Slope^+(F_t(z))$ is a segment. Moreover, the parabolic type
semigroup generated by $f_n$, where $f_n(z)=f(z)^{\frac{1}{n+1}},$
has an analytic extension to a sector with angle close to $\pi$,
and $Slope^+(F_t(z))$ is not a singleton.
\end{remark}

\section{Semigroups of composition operators}
\setcounter{equation}{0}

In this section, for a given semigroup
${\mathbf{F}=\{F_{t}\}_{t\geq 0}\subset\Hol(\Pi)}$, we study the
semigroup of composition operators $C_t: \phi\mapsto \phi\circ
F_t$ on Hardy spaces $H^p(\Pi), \, p \ge 1$, and solve the problem
of its analytic extension.

It was shown in \cite[Theorem 3.3]{Arv} that if $f$ generates a
semigroup $\mathbf{F}\subset\Hol(\Pi)$ which induces a semigroup
$\mathbf{T}$ of bounded composition operators on $H^p(\Pi)$, then
$\mathbf{T}$ is generated by the operator $\Gamma:\varphi\mapsto
\varphi'f$. From the point of view of the Lumer--Phillips theorem,
our first result is, in a sense, a converse assertion.

\begin{theorem}\label{th-flow-dissip}
Let $f\in \Hol(\Pi,\C)$ be such that $\dfrac{f(z)}{z+1}$ is
bounded in $\{z:\ \Re z>0,\ |z|>R\}$  for some $R>0$. Fix
$p\in(1,\infty)$ and define the linear operator $\Gamma:=\Gamma_p$
on the domain $\mathcal{D}_p:=\left\{\varphi\in H^p(\Pi):\
\varphi'\cdot f\in H^p(\Pi) \right\}$ by $\Gamma\varphi(z)
=\varphi'(z) f(z)$. If $\Gamma$ is dissipative, then $f$ is a
semigroup generator in $\Pi$. If, in addition, $I-\Gamma$ is
surjective, then $f$ generates a semigroup with the Denjoy--Wolff
point at $\infty$.
\end{theorem}

\begin{proof}
Recall that the dissipativity of the operator $\Gamma$ means that
for each $\varphi\in\mathcal{D}_p$,  $\|\varphi\|=1$, there exists
a support functional $\varphi^*$ such that $ \Re\langle
\Gamma\varphi,\varphi^*\rangle \le0.$ Note that in the space
$H^p(\Pi),\, p>1,$ the support functional is unique and can be
realized as follows:
\begin{equation}\label{functional}
\varphi^*(\psi)=\langle \psi,\varphi^*\rangle =\frac1\pi
\int_{-\infty}^\infty \psi(iy)
\frac{|\varphi(iy)|^p}{\varphi(iy)}\,dy.
\end{equation}

For $a\in\Pi\setminus\{1\}$ and $n\in\N$, set $\varphi_{
n}(z)=\dfrac{(z-a)^n}{(z+1)^{\frac2p}(z+\overline{a})^n}\,.$ First
we show that $\varphi_n\in\mathcal{D}_p$ and $\|\varphi_n\|=1$ for
each $n$. Indeed,
\[
|\varphi_n(z)|^p=\dfrac1{|z+1|^2}\left|\dfrac{z-a}{z+\overline{a}}
\right|^n \le \dfrac1{|z+1|^2} \,,\quad
|\varphi_n(iy)|^p=\dfrac1{y^2+1}\,.
\]
Therefore $\varphi_n\in H^p(\Pi)$ and $\|\varphi_n
\|^p=\frac1\pi\int_{-\infty}^\infty|\varphi_n(iy)|^pdy=1$. In
addition,
\begin{eqnarray*}
\varphi_n'(z)f(z)&=&\left(f(z)\frac{\varphi_n'(z)}{\varphi_n(z)}\right)
\varphi_n(z) \\
&=& f(z)\left(\frac n{z-a} -\frac2{p(z+1)} - \frac
n{z+\overline{a}} \right)\varphi_n(z).
\end{eqnarray*}
The boundedness of  $\dfrac{f(z)}{z+1}$ implies that $\varphi_n'f
\in H^p(\Pi)$, so $\varphi_n\in\mathcal{D}_p$.

We now calculate $\langle \Gamma\varphi_{n},\varphi_{n}^*\rangle$
directly. Using \eqref{functional}, we have
\begin{eqnarray*}
\langle \Gamma\varphi_n,\varphi_n^*\rangle &=& \langle
\varphi_n'f,\varphi_n^*\rangle = \frac1\pi \int_{-\infty}^\infty
f(iy)\frac{\varphi_n'(iy)}{\varphi_n(iy)}\, |\varphi_n(iy)|^pdy \\
&=& \frac1\pi \int_{-\infty}^\infty f(iy) \left(\frac n{iy-a}
-\frac2{p(iy+1)} - \frac n{iy+\overline{a}}
\right)\frac{dy}{y^2+1}
\end{eqnarray*}
Applying (for instance) the inverse Cayley transform and the
Cauchy formula, we obtain
\begin{eqnarray*}
\langle \Gamma\varphi_n,\varphi_n^*\rangle &=& f(1)\left(-\frac1p-
\frac n{1+\overline{a}} - \frac n{a-1}
\right) +f(a)\frac {2n}{a^2-1} \\
&=& \frac{2n}{|a-1|^2}\left[ f(a)\frac{\overline{a}-1}{a+1} -
f(1)\left(\Re a\cdot \frac{\overline{a}-1}{\overline{a}+1}\right)
\right]-\frac{f(1)}{p} .
\end{eqnarray*}
Since $\Re \langle \Gamma\varphi_n,\varphi_n^*\rangle\le0$ for
arbitrarily large $n$, it follows that
\[
\Re\left[ f(a)\frac{\overline{a}-1}{a+1} - f(1) \Re a\cdot
\frac{\overline{a}-1}{\overline{a}+1} \right] \le 0.
\]
Thus, by Proposition~\ref{propo-flow-inv}, $f$ is a semigroup
generator.

If $I-\Gamma$ is surjective, then by the Lumer--Phillips theorem,
$\Gamma$ generates a semigroup $\mathbf{T}=\{T(t),\, t\ge0\}$ with
$\|T(t)\|\le1$ for all $t\ge0$. On the other hand, by
\cite[Theorem~2.8]{Arv}, $\|T(t)\|=e^{-\delta t/p}$, where
$\delta$ is the angular derivative of $f$ at $\infty$. Therefore,
$\delta\ge0$. Hence, $f$ generates a semigroup with the
Denjoy--Wolff point at $\infty$.
\end{proof}
Next we establish a characterization of composition operators on
the Hardy spaces (cf. \cite[Theorem 2.5]{A-C-P}). For
$n=0,1,2,\ldots$, define $e_n\in H^p(\Pi)$ by
$e_n(z)=\dfrac{(z-1)^n}{(z+1)^{\frac2p+n}}\,.$

\begin{theorem}\label{th-charact-comp-operator}
Let $1\leq p<\infty$, and let $T$ be a bounded linear operator on
the Hardy space $H^p (\Pi)$. Then $T$ is a composition operator if
and only if all of the following three conditions hold:
\begin{itemize}
\item[{\rm(a)}] the function $F:=\dfrac1{(Te_0)^{\frac p2}}-1$ is
well defined and belongs to $\Hol(\Pi)$;

\item[{\rm(b)}] the angular derivative
$\angle\lim\limits_{z\to\infty}\dfrac{F(z)}z$ exists and is
different from zero;

\item[{\rm(c)}] the operator $T$ satisfies $Te_n= Te_0 \cdot
\left(1-2\left(Te_0\right)^{\frac p2} \right)^n$ for all $n \in
\N$.
\end{itemize}
Moreover, if these conditions hold, then $T\phi=\phi\circ F$ for
all $\phi\in H^p(\Pi)$.
\end{theorem}

\begin{proof}
Suppose that $T$ is a composition operator, say, $T\phi=\phi\circ
G$ for some $G\in\Hol(\Pi)$. Then $F=\dfrac1{(e_0\circ G)^{\frac
p2}}-1=G$, so condition (a) holds. Since $T$ is bounded, condition
(b) is also satisfied (see \cite{E-J-12}). The verification of
condition (c) is direct.

Conversely, suppose that conditions (a), (b) and (c) hold. We wish
to show that $T\phi=\phi\circ F$. Indeed, $e_0\circ F=Te_0$, and
\begin{eqnarray*}
&&e_n\circ F= \frac{(F(z)-1)^n}{(F(z)+1)^{\frac2p+n}}
=\frac{\left( \frac 1 {(Te_0)^{\frac p2}}-2\right)^n}
{\left(\frac{1}{Te_0}
\right) ^ {\frac p2\left(\frac2p+n\right)}} \\
&&=\left(Te_0 \right)^{1+\frac{pn}2}\left(\frac1{\left(Te_0
\right)^{\frac p2}} -2\right)^n =Te_0
\left(1-2\left(Te_0\right)^{\frac p2} \right)^n  =Te_n
\end{eqnarray*}
for $n=1,2,3\ldots$. Thus, the linearity and the continuity of
both $T$ and the operator of composition with $F$ imply that
$T\phi=\phi\circ F$ for all $\phi\in H^p(\Pi)$.
\end{proof}

Now we are ready to present the main result of this section.

\begin{theorem}\label{th-operator-sem}

Let $f\in \Hol(\Pi,\C)$ be such that $\Re f(z)>0$ for all
$z\in\Pi,$ and ${\lim\limits_{z \to \infty}\dfrac{f(z)}{z+1}=0}$.
Fix $p\in(1,\infty)$ and define the linear operator
$\Gamma:=\Gamma_p$ on the domain $\mathcal{D}_p:=\left\{\varphi\in
H^p(\Pi):\ \varphi'\cdot f\in H^p(\Pi) \right\}$ by
$\Gamma\varphi(z) =\varphi'(z) f(z)$. For all
$\theta_1,\theta_2>0$ such that $\theta_1+\theta_2\leq \pi$, the
following assertions are equivalent:
\begin{itemize}
\item[{\rm(i)}] the function $f$ satisfies inequality
\eqref{gen-in-sector}, that is,
\[
-\frac{\pi}{2}+  \theta_1\leq\arg f(z) \leq \frac{\pi}{2}-
\theta_2,  \quad z \in \Pi;
\]

\item[{\rm(ii)}] the operators $e^{-i\theta_1}\Gamma$ and
$e^{i\theta_2}\Gamma$ generate semigroups of contractions on the
space $H^p(\Pi)$;

\item[{\rm(iii)}] the semigroup $\{T(t), t\geq 0 \}$ generated by
$\Gamma$ can be analytically extended to the semigroup
$\{T(\zeta),\zeta \in \Omega(\theta_1,\theta_2)\}$ consisting of
composition operators with $\|T(\zeta)\|=1$ on $H^p(\Pi)$ for all
$\zeta \in \Omega(\theta_1,\theta_2)$.

\end{itemize}

\end{theorem}

\begin{proof}
First we note that the operator
$\Gamma^{\theta}:=e^{i\theta}\Gamma$ has the same domain
$\mathcal{D}_p$ for every $\theta \in \R$.

{\it Step 1.} Suppose that assertion (i) holds. For $\theta \in
[-\theta_1,\theta_2]$, let $f_\theta(z)=e^{i\theta}f(z)$. By
assumption, $\Re f_\theta(z)\geq0$ and $\lim\limits_{z \to \infty}
\frac{f_\theta(z)}{z+1}= 0$. Therefore, $f_\theta$ generates a
one-parameter parabolic type semigroup $\{F^\theta_t\}_{t\geq 0}
\subset \Hol(\Pi)$ with the Denjoy--Wolff point at $\infty$. Thus,
the composition operators $C^\theta_t$ defined by
$C^\theta_t\phi:=\phi\circ {F^\theta_t}$ are bounded on
$H^p(\Pi)$, and $\|C^\theta_t\|=1$ for all $t\geq 0$ (see
\cite[Theorem~2.8]{Arv} \ and \cite[Corollary~3.5]{E-J-12}). By
\cite[Theorem~3.3]{Arv}, the infinitesimal generator of the
contractive semigroup $\{C^\theta_t\}$ is $\Gamma^{\theta}$. So,
assertion (ii) follows (notice that the implication
(i)$\Rightarrow$(iii) is proved too).

{\it Step 2.} Assume now that (ii) holds. By the Lumer--Phillips
theorem, the operators $\Gamma^{-\theta_1}=e^{-i\theta_1}\Gamma$
and $\Gamma^{\theta_2}=e^{i\theta_2}\Gamma$ are dissipative; and
for each $\lambda >0,$ the operators $\lambda I-
\Gamma^{-\theta_1}$ and $\lambda I- \Gamma^{\theta_2}$ are
surjective. Thus both the resolvent sets of $\lambda I-
\Gamma^{-\theta_1}$ and  of $\lambda I- \Gamma^{\theta_2}$ contain
the right half-plane (see Remark 5.8 in \cite{Pa}). Therefore, for
every $\theta \in
[-\frac{\pi}{2}-\theta_1,\theta_2+\frac{\pi}{2}]$, the operator
$I-e^{i\theta}\Gamma$ is surjective. Moreover,
\begin{equation*}
\Re e^{-i\theta_1}\langle  \Gamma \varphi,
\varphi^*\rangle=\Re\langle e^{-i\theta_1} \Gamma \varphi,
\varphi^*\rangle\le 0
\end{equation*}
for some support functional $\varphi^*$ at the point $\varphi\in
H^p(\Pi)$, which yields
\begin{equation*}
\frac{\pi}{2}+\theta_1 \leq  \arg\langle \Gamma \varphi,
\varphi^*\rangle \leq \frac{3\pi}{2}+\theta_1.
\end{equation*}
Similarly,
\begin{equation*}
\frac{\pi}{2}-\theta_2 \leq  \arg\langle \Gamma \varphi,
\varphi^*\rangle  \leq \frac{3\pi}{2}-\theta_2.
\end{equation*}
Therefore,
\begin{equation*}
\frac{\pi}{2}+\theta_1 \leq  \arg\langle \Gamma \varphi,
\varphi^*\rangle \leq \frac{3\pi}{2}-\theta_2,
\end{equation*}\label{angl-dissipative}
which means that
\begin{equation}
\Re\langle e^{i\theta}\Gamma \varphi, \varphi^*\rangle \leq 0
\end{equation}
for all $\theta \in [-\theta_1,\theta_2]$.

Consequently, for each $\theta \in [-\theta_1,\theta_2]$, the
operator $\Gamma^\theta=e^{i\theta} \Gamma $ is dissipative, and
${I-\Gamma^\theta}$ is surjective. The Lumer--Phillips theorem
then implies that each $\Gamma^\theta$ generates a one-parameter
semigroup of contractions. Moreover, by
\cite[Theorem~2.4]{Arendt}, the contractive semigroup $\{T(t),\
t\ge0\}$ generated by $\Gamma$ can be extended to the semigroup
$\{T(\zeta)\}$ analytic in $\zeta \in \Omega (\theta_1,\theta_2)$.

Note also that by our assumptions, $f$ generates a parabolic type
semigroup $\mathbf{F}=\{F_t\}_{t \geq 0}$ which, in turn, induces
the semigroup of composition operators $\mathbf{T}=\{T(t),t \geq
0\}$ defined by $T(t)\varphi= \varphi\circ F_t$ and generated
by~$\Gamma$. For each $n \in \N$, the function ${\ell_n :
\Omega(\theta_1,\theta_2) \mapsto H^p(\Pi)}$ defined by
\[
\ell_n(\zeta)= T(\zeta)e_n -
T(\zeta)e_0\left(1-2\left(T(\zeta)e_0\right)^\frac{p}{2}\right)^n
\]
is analytic in $\Omega(\theta_1,\theta_2)$. Since $\{T (t),t \geq
0\}$ is a composition semigroup, each $\ell_n$ vanishes on $\R^+$
by Theorem~\ref{th-charact-comp-operator}. By the uniqueness
theorem, $\ell_n\equiv0$ in $\Omega(\theta_1,\theta_2)$. Hence, by
Theorem~\ref{th-charact-comp-operator}, $T(\zeta)$ is a
composition operator for each $\zeta\in\Omega(\theta_1,\theta_2)$.
This proves (ii)$\Rightarrow$(iii). (Note in passing that the
implication (iii)$\Rightarrow$(ii) is obvious.)

{\it Step 3.} Assertion (iii) implies that there exists a
semigroup $\widetilde{\mathbf{F}}=\{\widetilde{F}_t\}_{t \geq
0}\subset\Hol(\Pi)$ that induces a semigroup of composition
operators $\{\widetilde{C}_t\}$ generated by
$\Gamma^{-\theta_1}:=e^{-i\theta_1}\Gamma$. The inequality
\begin{equation*}
\left|\phi\circ \widetilde{F}_t(z) -\phi(z)\right| \le
\frac{\left\|\widetilde{C}_t\phi-\phi\right\|} {(4\pi\Re z)^{1/p}}
\end{equation*}
(cf. Lemma 3.2 in \cite{Arv}) implies that
$\widetilde{\mathbf{F}}$ is a continuous (hence, generated)
semigroup on $\Hol(\Pi)$. Denote its generator by $\widetilde{f}$.

Since the norm of composition operators $\widetilde{C}_t$ equals
$1$ for all $t$, these operators are bounded on $H^p(\Pi)$. By
\cite[Theorem 3.3]{Arv}, the infinitesimal generator of
$\{\widetilde{C}_t\}$ is $\Gamma^{-\theta_1}\varphi =\varphi'
\widetilde{f}$. On the other hand, $\Gamma^{-\theta_1}\varphi
=e^{-i\theta_1} \Gamma\varphi=e^{-i\theta_1} \varphi' f,$ so
$\widetilde{f}(z)=e^{-i\theta_1}f(z)$. Thus, $\lim\limits_{z \to
\infty}\frac{\widetilde{f}(z)}{z+1}=0$, that is,
$\widetilde{\mathbf{F}}$ is of parabolic type with the
Denjoy--Wolff point at $\infty$. Therefore, $\Re
\widetilde{f}(z)=\Re e^{-i\theta_1}f(z)\geq0$. Similarly, $\Re
e^{i\theta_2}f(z)\ge 0$. These last two inequalities imply (i).
\end{proof}

In fact, the equivalence of assertions (ii) and (iii) to assertion
(i) implies the following fact.

\begin{corol}
If either assertion (ii) or (iii) of Theorem \ref{th-operator-sem}
holds for some $p>1$, then both assertions hold for every $p>1$.
\end{corol}

\begin{remark} We have already seen in the proof of Theorem \ref{th-operator-sem}
that if (ii) holds, then the operator $e^{i\theta} \Gamma$ is
dissipative for every $\theta \in [-\theta_1, \theta_2]$. Thus, by
Theorem \ref{th-flow-dissip}, each function $e^{i \theta}f$ is a
semigroup generator. In addition, by the assumptions of Theorem
\ref{th-operator-sem}, $\lim\limits_{z \to
\infty}\dfrac{e^{i\theta}f(z)}{z+1}=0$. Therefore, by Proposition
\ref{prop-DW}, $\Re e^{i \theta} f(z) \geq 0$ for all $z \in \Pi$.
This shows independently that (i) follows from (ii).
\end{remark}

The next result follows from Theorems
\ref{thm-ector-angle-general} and~\ref{th-operator-sem}.

\begin{corol}\label{thm-operator-sector-angle-general}
Let $\gamma_1, \gamma_2:\R^+ \to (0,\frac\pi2)$ be decreasing
functions. Assume that $f\in\Hol(\Pi,\C)$ satisfies
\[
-\gamma_1(\Re z)\leq \arg f(z)\leq \gamma_2(\Re z)\quad \mbox{for
all}\quad z\in\Pi.
\]
Let $1< p<\infty$.  For each $k>0$, define the linear operator
$\Gamma^k$ on the domain $\mathcal{D}_{p,k}:=\left\{\varphi\in
H^p(\Pi):\ \varphi'f(\cdot+k)\in H^p(\Pi) \right\}$ by
$\Gamma^k\varphi(z) =\varphi'(z) f(z+k)$. Then $\Gamma^k$
generates a semigroup of composition operators which can be
analytically extended to the sector
$\Omega\left(\frac\pi2-\gamma_1(k),\frac\pi2-\gamma_2(k)\right)$.
In particular, if $\lim\limits_{s \to \infty} \gamma_1(s)=
\lim\limits_{s \to \infty} \gamma_1(s)=0,$ then for each
$\varepsilon>0$, there exists a $k>0$ such that the composition
semigroup generated by $\Gamma^k$ is analytic in the sector
$\Omega\left(\frac\pi2-\varepsilon, \frac\pi2-\varepsilon\right).$
\end{corol}

\begin{examp}
Let us return to the function $f(z)=\dfrac{z+a}{z+b}\,,\ 0<a<b$,
considered in Example~\ref{exam3}. For $k>0$, consider now the
linear operator $\Gamma^k,\ k\ge0,$ defined by
$\Gamma^k:\varphi(z)\mapsto {\varphi'(z)f(z+k)} =\varphi'(z)
\dfrac{z+k+a}{z+k+b}\,.$ By
Corollary~\ref{thm-operator-sector-angle-general}, $\Gamma_k$
generates a semigroup of composition operators which can be
analytically extended to the sector
$\Omega\left(\frac\pi2-\gamma(k),\frac\pi2-\gamma(k)\right)$,
where $\gamma(k)=\arctan\dfrac{b-a}{2\sqrt{(k+a) (k+b)}}\,$ (see
Example~\ref{exam3} above). Similarly to as in that example, for
each $\varepsilon>0$ and $k>\dfrac{b-a-(a+b)\sin\varepsilon}
{2\sin\varepsilon}$, the composition semigroup generated by
$\Gamma^k$ has an analytic extension to the sector
$\Omega\left(\frac\pi2-\varepsilon, \frac\pi2-\varepsilon\right).$
\end{examp}

\begin{corol}\label{last}
Let $\mathbf{T}$ be a $C_0$-semigroup of composition operators on
$H^p(\Pi)$  induced by a semigroup $\mathbf{F}\subset\Hol(\Pi)$.
If $\mathbf{T}$ can be analytically extended to sector, then
$\mathbf{F}$ is of parabolic type.
\end{corol}

\begin{proof}
Since the elements of $\mathbf{T}$ are bounded, the elements of
$\mathbf{F}$ have finite angular derivatives at $\infty$.
Therefore, the generator $f$ of $\mathbf{F}$ has also a finite
angular derivative $\delta$ at $\infty$, and this derivative is
known to be real (cf. \cite{SD}). It follows by results in
\cite{A-C-P} and Theorem~\ref{th-flow-dissip} above that there
exists small enough $\theta>0$ such that the function
$f_\theta:=e^{i\theta}f$ is also a semigroup generator. The
angular derivative of $f_\theta$ at $\infty$ equals
$e^{i\theta}\delta$, which is real only if $\delta=0$. So,
$\mathbf{F}$ is of parabolic type.
\end{proof}

\vspace{2mm}


\end{document}